\documentclass[11pt]{article}

\usepackage{graphicx}
\usepackage[a4paper,margin=1in]{geometry}
\usepackage{amsmath,amssymb,amsthm,mathtools}
\usepackage{enumitem}
\usepackage{booktabs}
\usepackage{xcolor}
\usepackage{hyperref}
\usepackage{tikz}
\usepackage[affil-it]{authblk}

\hypersetup{
	colorlinks=true,
	linkcolor=blue!60!black,
	citecolor=blue!60!black,
	urlcolor=blue!60!black
}

\newtheorem{theorem}{Theorem}[section]
\newtheorem{lemma}[theorem]{Lemma}

\newtheorem{corollary}[theorem]{Corollary}
\newtheorem{definition}[theorem]{Definition}
\newtheorem{remark}[theorem]{Remark}

\DeclareMathOperator{\mex}{mex}

\title{Asymptotic Bounds for Irredundant Covers of Square Grids by \(2\times2\) Cards}
\author{ Aksel Eruysal\\
	\texttt{akseleruysal@gmail.com}\\
	\texttt{ İzmir, Turkey}
	\and 
	S. Kaan G\"{u}rb\"{u}zer\\
	\texttt{kaan.gurbuzer@deu.edu.tr}\\
	\texttt{Department of Mathematics, Dokuz Eyl\"{u}l University \\ 
		İzmir, Turkey}
}	
\date{}

\begin{document}
	
	\maketitle
	
	\begin{abstract}
		
		
		We study an irredundant covering of an $m \times m$ square grid by $2 \times 2$ cards, where cards may overlap and every card must lie entirely inside the grid. Every square of the grid must be covered, and every card must contain at least one square that is covered only by itself. Let \(F(m)\) denote the maximum number of cards in such an irredundant cover. We develop a general framework for bounding the maximum possible size of such an irredundant cover based on square multiplicities, double counting, and local geometric restrictions arising around highly covered squares. These arguments yield a general upper bound whose leading term is $\frac{1}{2}m^2$, accompanied by a negative linear term. In the opposite direction, we construct a family of irredundant covers based on a density-$\frac{1}{2}$ covering pattern, referred to as the period-four staircase pattern, together with a boundary-repair construction. Thus, we obtain a lower bound with the same leading term and a linear-order error, proving that for every \(m\ge6\),
		\[
		\frac{2m}{9}
		\le
		\frac12m^2-F(m)
		\le
		2m-3,
		\]
		and consequently that the maximum density of an irredundant cover tends to $\frac{1}{2}$ as $m\to\infty$. The known $10\times 10$ case is used as motivation and as a finite benchmark for the general theory.

	\end{abstract}
	
	\section{Introduction}\label{sec:Intro}
	
	In this paper, we develop a general framework for bounding the maximum number of
	\(2\times2\) cards in an irredundant cover of a square grid. An irredundant cover
	by \(2\times2\) cards must satisfy the following properties: cards may overlap;
	however, every card must contain at least one private square, that is, a square
	covered by that card and by no other card, and every square of the grid must be
	covered.
	
	We define \(F(m)\) to be the maximum number of \(2\times2\) cards in an
	irredundant cover of an \(m\times m\) grid. We bound \(F(m)\) using the
	multiplicities of the grid squares, which partition the squares according to
	the number of cards covering them, together with other combinatorial arguments
	arising from the geometry of the grid and the irredundancy condition.
	In particular, by analyzing these multiplicity classes, we obtain several local
	restrictions that can be used to bound \(F(m)\).
	
	The fundamental difficulty of the problem arises from the tension between
	preserving irredundancy and maximizing the number of cards used. Adding new
	cards may increase the size of the cover; however, it also increases overlap
	and may cause private squares to be lost, thereby destroying irredundancy.
	A significant part of the analysis in this paper therefore focuses on
	understanding how these two opposing effects are quantitatively balanced.
	
	From this perspective, the problem may be viewed as a geometric and highly
	structured instance of the maximum minimal set-cover problem, which has been
	studied more generally in the combinatorial optimization literature. Indeed,
	a cover is irredundant precisely when it is inclusion-wise minimal: if a
	selected card can be removed without destroying coverage, then that card has
	no private square; conversely, any card possessing a private square is
	indispensable to the cover. Maximum minimal set cover, together with its dual
	formulation as the maximum minimal hitting set problem, has previously been
	studied in general set systems and hypergraphs; see, for example,
	\cite{berge1989,eitergottlob1995,araujoetal2023,damascenoetal2026}.
	
	This connection can be expressed more explicitly in the language of hypergraph
	transversals. Let \(\mathcal C_m\) denote the set of all possible card
	positions, and define the dual board hypergraph \(\mathcal H_m\) by taking the
	elements of \(\mathcal C_m\) as its vertices and, for each square \(x\in G_m\) (with \(G_m\) defined in
	Section~\ref{sec:definitions}), defining the hyperedge

	\[
	E_x=\{C\in\mathcal C_m:x\in C\}.
	\]
	Then a family \(U\subseteq\mathcal C_m\) covers the grid if and only if \(U\)
	is a transversal of \(\mathcal H_m\). Similarly, \(U\) is an irredundant cover
	if and only if it is a minimal transversal: a card \(C\in U\) has a private
	square precisely when there exists some \(x\in G_m\) such that
	\[
	E_x\cap U=\{C\}.
	\]
	Thus, our problem can be reformulated as
	\[
	F(m)=\Upsilon(\mathcal H_m),
	\]
	where \(\Upsilon\) denotes the upper transversal number, that is, the maximum
	cardinality of a minimal transversal of a hypergraph. General results
	concerning upper transversals can be found in
	\cite{henningyeo2018,henningyeo2019,henningyeo2020}, while related algorithmic
	and dualization problems are studied in
	\cite{fredmankhachiyan1996,borosetal2001,damaschke2011}.
	
	However, the hypergraph considered here is not an arbitrary rank-four
	hypergraph. The hypergraph \(\mathcal H_m\) has a highly rigid local structure
	arising from the geometric arrangement of \(2\times2\) squares on a grid,
	together with additional restrictions near the boundary. General upper-transversal results do not directly exploit the
	boundary cancellation and local grid restrictions used here.
	Using this geometric structure, we establish
	$F(m)=\frac12m^2-\Theta(m)$, determining the leading asymptotic
	coefficient and showing that the deficit from $\frac12m^2$
	has linear order. The private-square
	condition is also closely analogous to the private-neighbour characterization
	used for irredundant and minimal dominating sets in graph theory
	\cite{cockaynehedetniemi1978,bazganetal2018}. Related geometric covering and
	packing problems involving squares have been studied in
	\cite{elkhechenetal2009,biedletal2021,baloghetal2025}, while exact
	computational work on grid graphs provides a useful methodological comparison
	\cite{alankoetal2011}.
	
	An important finite instance of the problem, namely the \(10\times10\) case,
	previously appeared as Problem 37 of the 56th Leningrad City Mathematical
	Olympiad \cite{lmo1990archive,vasilievfomin1990}. For this instance, the optimal
	value is known to be
	\[
	F(10)=39.
	\]
	The cited solution provides both an irredundant cover using \(39\) cards
	and an upper-bound argument showing that every cover using \(40\) or
	more cards must be redundant \cite{fomin2025}.

	The purpose of this paper is to analyze the general \(m\times m\)
	square grid systematically. The main question we focus on is
	how the maximum size of an irredundant cover scales as \(m\)
	gets larger.
	
	
	First, we obtain a defect identity using the multiplicities of the squares,
	
	$$
	5|U|=2m^2-D_U,
	$$
	
	where \(|U|\) denotes the number of selected cards and \(D_U\) records the net contribution of surplus private squares and regions of high covering multiplicity.
	
	Then, using the surplus of private boundary squares together with an interior multiplicity count, we derive a boundary-cancellation identity that leads to our main upper bound. Using the local restrictions around the squares covered four times, we derive, for \(m\geq4\), the bound
	
	$$
	F(m)\leq
	\left\lfloor
	\frac{3m^2-12m+16\left\lfloor 2m/3\right\rfloor}{6}
	\right\rfloor,
	$$
	
	which in particular implies
	
	$$
	F(m) \leq \frac{1}{2}m^2 - \frac{2}{9}m.
	$$
	
	In the opposite direction, we construct a family of irredundant covers using a period-four staircase pattern together with a boundary repair construction. This gives, for \(m\geq6\), the lower bound
	
	$$
	F(m) \geq \frac{1}{2}m^2 - 2m + 3.
	$$
	
	Combining the upper and lower bounds, we show that
	
	$$
	F(m) = \frac{1}{2}m^2 - \Theta(m),
	$$
	
	and therefore
	
	$$
	\lim_{m\to\infty}\frac{F(m)}{m^2}=\frac{1}{2}.
	$$
	
	Thus, although the exact value of \(F(m)\) is not determined for every \(m\), the leading asymptotic behavior of the problem is determined.

	\section{Preliminaries}
	\label{sec:definitions}
	In this section, we introduce the notation and definitions used throughout the paper. First, we define the square grid for \(m\ge2\) as the Cartesian product $\{1,\ldots,m\}\times\{1,\ldots,m\}$ and denote it by $G_m$. We represent each card using its upper-left corner square. For each pair $(a,b)$ of integers, where \(1\le a,b\le m-1\), the card $C_{a,b}$ is defined by the set \(\{(a,b),(a+1,b),(a,b+1),(a+1,b+1)\}\) of pairs of integers. Then, the set containing all possible cards can be given as follows.
	\[
	\mathcal C_m=\{C_{a,b}:1\le a,b\le m-1\}
	\]
	Hence, $|\mathcal C_m|=(m-1)^2.$ The illustration of a card is given in Figure \ref{fig:card}.
	\begin{figure}[h]
		\centering
		\begin{tikzpicture}[scale=1.75]
			\draw[step=1,gray!50,thin] (0,0) grid (4,4);
			\fill[blue!12] (1,1) rectangle (3,3);
			\draw[blue!70!black,very thick] (1,1) rectangle (3,3);
			\node at (1.5,2.5) {\(\scriptstyle(a,b)\)};
			\node at (2.5,2.5) {\(\scriptstyle(a,b+1)\)};
			\node at (1.5,1.5) {\(\scriptstyle(a+1,b)\)};
			\node at (2.5,1.5) {\(\scriptstyle(a+1,b+1)\)};
			\node[blue!70!black] at (2,-0.45) {\(C_{a,b}\)};
		\end{tikzpicture}
		\caption{The \(2\times2\) card \(C_{a,b}\).}
		\label{fig:card}
	\end{figure}
	
	\begin{definition}
		The \text{boundary squares} of the grid $G_m$ are:
		\[
		\partial G_m = \{(i,j) \in G_m : i \in \{1,m\} \text{ or } j \in \{1,m\}\}.
		\]
		The set of \text{interior squares} is denoted by $I_m$ and it is equal to the set $G_m \setminus \partial G_m$. So, a card $C_{a,b}$ is called a boundary card if and only if $a \in \{1, m-1\}$ or $b \in \{1, m-1\}.$
	\end{definition}

	\begin{definition}
		Any non-empty subset $U$ of $\mathcal C_m$ is called a configuration.
	\end{definition}
	
	\begin{definition}
		For a given configuration $U,$ the multiplicity $\mu(x)$ of a square \(x\in G_m\) is the number of cards $C_{a,b}\in U$ containing $x,$ that is,
		\[
		\mu(x)=|\{C_{a,b}\in U:x\in C_{a,b}\}|.
		\]
		
		A square \(x\) is private for a selected card \(C_{a,b}\in U\) if $x\in C_{a,b}$
		and $\mu(x)=1.$
	\end{definition}
	
	\begin{definition}
		The configuration \(U\) is called  an irredundant cover of \(G_m\) if
		\begin{enumerate}[label=(\roman*)]
			\item \(\mu(x)\ge1\) for every \(x\in G_m\);
			\item every selected card \(C_{a,b}\in U\) contains at least one private square.
		\end{enumerate}
	\end{definition}
	
	As stated in Section \ref{sec:Intro}, the function $F:\mathbb{Z}_{>1} \rightarrow \mathbb{Z}_{>0}$ maps a positive integer $m>1$ to the maximum cardinality among irredundant covers of $G_m,$ that is,
	\[
	F(m)=\max\{|U|: U\subseteq\mathcal C_m \text{ is an irredundant cover of }G_m\}
	\]
	
	For a selected configuration, we partition the squares in $G_m$ with respect to their multiplicities as follows:
	\[
	S_j=\{x\in G_m:\mu(x)=j\},
	\qquad
	s_j=|S_j|.
	\]
	We observe that there can be at most $4$ cards that cover a single grid square due to the geometry of the cards. Hence, for an irredundant cover, the multiplicity $j$ can be at most $4$. Now, let $U$ be a given irredundant cover of $G_m$. The number $p_U$ counts how many private squares remain after assigning
	one private square to each selected card, which is precisely $s_1-|U|.$

	\begin{definition}
		The defect $D_U$ of an irredundant cover $U$ is defined as \( p_U-(s_3+2s_4). \)
	\end{definition}
	The role of this definition will become clear in Theorem \ref{D_U theo}, where it yields an exact identity relating \(D_U\) and \(|U|\).

	At the end of this section, we consider an important situation involving squares of multiplicity four. Let \(x=(i,j)\in S_4\).  A boundary square can be covered by at most two cards, so \(x\) must be an interior square and \(2\le i,j\le m-1\).  Since  $x$ is covered four times, all four possible cards containing \(x\) are selected, namely they are  \(C_{i-1,j-1}, C_{i-1,j}, C_{i,j-1}, C_{i,j}.\) Hence, we call the set
	\[
	\Phi(x)=\{i-1,i,i+1\}\times\{j-1,j,j+1\}
	\]
	the \(3\times3\) footprint of \(x\).
	
	\begin{figure}[h]
		\centering
		\begin{tikzpicture}[scale=0.75]
			\draw[step=1,gray!50,thin] (0,0) grid (3,3);
			\foreach \x/\y/\t in {0/2/1,1/2/2,2/2/1,0/1/2,1/1/4,2/1/2,0/0/1,1/0/2,2/0/1} {
				\node at (\x+0.5,\y+0.5) {\(\t\)};
			}
			\node[blue!70!black] at (0.5,3.35) {\(\scriptstyle p_{NW}\)};
			\node[blue!70!black] at (2.5,3.35) {\(\scriptstyle p_{NE}\)};
			\node[blue!70!black] at (0.5,-0.35) {\(\scriptstyle p_{SW}\)};
			\node[blue!70!black] at (2.5,-0.35) {\(\scriptstyle p_{SE}\)};
			\draw[blue!70!black,very thick] (0,0) rectangle (3,3);
		\end{tikzpicture}
		\caption{Multiplicity pattern in an \(S_4\)-footprint contributed by the $4$ cards containing the \(S_4\) square.  The four corners are
			the only possible private squares for the four cards containing the center.}
		\label{fig:s4-footprint}
	\end{figure}
	

	\section{The defect framework and boundary upper bounds}
	
	In this section, we develop the main framework used to obtain our stronger upper bounds for \(F(m)\). We begin by counting grid squares according to their covering multiplicities and derive the defect identity, which expresses the number of selected cards in terms of the surplus of private squares and the contribution of highly covered squares. We then incorporate the geometry of the boundary, where the possible card configurations are more restricted, and combine the resulting boundary private-square surplus with an interior area count. This produces a boundary-cancellation identity in which the contribution of triple-covered squares disappears. We first control the remaining fourfold-covered squares through a packing argument, obtaining an intermediate bound. Finally, we use the local structure of an \(S_4\)-footprint to relate fourfold-covered squares to nearby double-covered interior squares, yielding the stronger uniform upper bound.
	
	\subsection{The defect identity}
	
	\label{sec:defect-identity}
	
	Because a grid square can lie in at most four axis-aligned \(2\times2\) cards, every covered square must satisfy $1\le\mu(x)\le4.$
	
	The sets \(S_1,S_2,S_3,S_4\) partition \(G_m\), hence
	\begin{equation}
		s_1+s_2+s_3+s_4=m^2.
		\label{eq:grid-area}
	\end{equation}
	
	For the second equation, we count the total number of card-square incidences, that is, the total number of covered grid-square occurrences counted with multiplicity. Since each card \(C^i\) covers \(4\) squares and there are \(|U|\) selected cards, the total number of incidences is \(4|U|\). On the other hand, counting the same incidences using the square multiplicities gives

	\begin{equation}
		s_1+2s_2+3s_3+4s_4=4|U|.
		\label{eq:incidence-general}
	\end{equation}
	
	The set \(S_1\) is exactly the set of private squares. Since every selected card has at least one private square and a private square belongs to exactly one selected card, we have
	\[
	s_1\ge |U|.
	\]
	
	\begin{theorem}
		For every irredundant cover of \(G_m\), with \(m\ge2\),
		\[
		{
			5|U|=2m^2-D_U,
			\qquad
			D_U=(s_1-|U|)-(s_3+2s_4).
		}
		\]
		\label{D_U theo}
	\end{theorem}
	
	\begin{proof}
		Subtract \eqref{eq:grid-area} twice from
		\eqref{eq:incidence-general}. This gives the following:
		\[
		4|U|-2m^2=-s_1+s_3+2s_4.
		\]
		Since \(s_1=|U|+p_U\), we get
		\[
		4|U|-2m^2=-(|U|+p_U)+s_3+2s_4.
		\]
		Therefore
		\[
		5|U|=2m^2-p_U+s_3+2s_4.
		\]
		Finally, by the definition of \(D_U\), we obtain \(5|U|=2m^2-D_U\).
	\end{proof}
	
	\begin{corollary}
		For the \(10\times10\) grid,
		\[
		5|U|=200-D_U,
		\qquad
		D_U=(s_1-|U|)-(s_3+2s_4).
		\]
	\end{corollary}
	
	\begin{remark}
		For \(m=10\), it suffices to prove that \(D_U>0\) in order to obtain the sharp upper bound \(|U|\le 39\). Indeed, since
		\[
		D_U=200-5|U|,
		\]
		we have \(D_U\equiv 0 \pmod{5}\). Thus, \(D_U>0\) implies \(D_U\ge 5\), and consequently
		\[
		200-5|U|\ge 5,
		\]
		which gives \(|U|\le 39\).
	\end{remark}
	This observation explains why the defect identity is useful: the sharp \(10\times10\) bound is equivalent to proving a small but positive global defect. However, one should not forget that the main focus is on the general \(m \times m\) bound in this paper.
	
	\subsection{Boundary cancellation and the first \texorpdfstring{\(S_4\)}{S4} bound}
	
	We now combine boundary private-square surplus with an interior area identity.
	Adding the two equations cancels the \(s_3\)-term, leaving only \(s_4\).  In
	this section \(s_4\) is controlled by a packing argument, which gives the
	intermediate asymptotic constant \(14/27\).  The following section improves
	this further by coupling \(S_4\)-squares to nearby double-covered interior
	squares.
	
	We begin by controlling the number of cards that can occur along the boundary of the grid. The irredundancy condition imposes a simple local restriction on consecutive boundary-card positions, which will give us an upper bound on the total number of selected boundary cards.
	
	A side-card overlap pair is a pair consisting of a
	selected card and a side of the grid touched by that card.  The four corner
	cards
	\[
	C_{1,1},\quad C_{1,m-1},\quad C_{m-1,1},\quad C_{m-1,m-1}
	\]
	are forced for \(m\ge3\), because the four corner grid squares have unique
	covering cards.  Each corner card has two side-card overlap pairs.  Every other
	boundary card touches exactly one side and has one side-card incidence.
	
	Let \(N_e\) denote the number of selected non-corner boundary cards, and let
	\(N_{\mathrm{int}}\) denote the number of interior cards. Then the following equation holds:
	\begin{equation}
		|U|=4+N_e+N_{\mathrm{int}}.
		\label{eq:N-split}
	\end{equation}
	
	\begin{lemma}[Line-of-three lemma]
		No three consecutive side-card positions can all be selected along the same side of the grid.
	\end{lemma}
	
	\begin{proof}
		Consider the top side. If
		\[
		C_{1,b},\qquad C_{1,b+1},\qquad C_{1,b+2}
		\]
		are all selected, then the middle card \(C_{1,b+1}\) has no private square. Its two left squares are covered by \(C_{1,b}\), while its two right squares are covered by \(C_{1,b+2}\). This contradicts irredundance. The same argument applies to the other three sides.
	\end{proof}
	
	\begin{figure}[h]
		\centering
		\begin{tikzpicture}[scale=1]
			\draw[step=1,gray!45,thin] (0,0) grid (7,2);
			\node[left] at (0,1.5) {\(\text{top side}\)};
			\foreach \x/\c in {1/blue!15,2/red!15,3/blue!15} {
				\fill[\c] (\x,0) rectangle (\x+2,2);
				\draw[very thick] (\x,0) rectangle (\x+2,2);
			}
			\node at (2,2.35) {\(\scriptstyle C_{1,b}\)};
			\node at (3,2.35) {\(\scriptstyle C_{1,b+1}\)};
			\node at (4,2.35) {\(\scriptstyle C_{1,b+2}\)};
			\node[red!70!black] at (3,-0.45) {middle card has no private square};
		\end{tikzpicture}
		\caption{Three consecutive side-card positions cannot all be selected.}
		\label{fig:side-three}
	\end{figure}
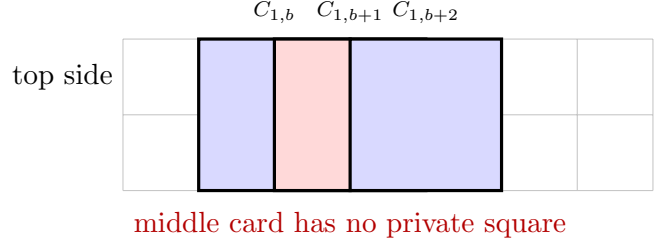
	
	\begin{lemma}[Number of non-corner boundary cards]
		For every \(m\ge3\),
		\[
		{
			N_e\le4\left\lfloor\frac{2m}{3}\right\rfloor-8.
		}
		\]
	\end{lemma}
	
	\begin{proof}
		On one side of \(G_m\), there are \(m-1\) possible side-card positions. By the line-of-three lemma, no three consecutive positions may all be selected. Therefore, among these \(m-1\) positions, at most
		\[
		\left\lceil\frac{2(m-1)}{3}\right\rceil
		=
		\left\lfloor\frac{2m}{3}\right\rfloor
		\]
		can be selected.
		The displayed identity can be obtained using residues modulo \(3\): if
		\(m=3q,3q+1,3q+2\), both sides are respectively \(2q,2q,2q+1\).
		
		Summing over the four sides gives at most
		\[
		4\left\lfloor\frac{2m}{3}\right\rfloor
		\]
		selected side-card incidences.  The four corner cards contribute \(8\), 2 per corner, while every non-corner boundary
		card contributes one incidence and one boundary card.  Hence
		\[
		\#\{\text{side-card overlap pairs}\}=8+N_e
		\]
		and therefore
		\[
		8+N_e
		\le
		4\left\lfloor\frac{2m}{3}\right\rfloor.
		\]
		Therefore
		\[
		N_e\le4\left\lfloor\frac{2m}{3}\right\rfloor-8.
		\]
	\end{proof}
	
	Having bounded the number of selected boundary cards, we now use the same boundary structure to count private squares. This allows us to express the contribution of the boundary to the private-square surplus in terms of \(N_e\).
	
	Let $s_{1,\mathrm{b}}$ be the number of boundary squares covered exactly once; in other words boundary squares which are also a private square. Boundary squares have multiplicity either \(1\) or \(2\), so the total boundary multiplicity is
	\[
	s_{1,\mathrm{b}}+2((4m-4)-s_{1,\mathrm{b}}).
	\]
	
	We count the same quantity by studying boundary cards as follows:\\
	A corner card covers
	three boundary squares: the grid corner itself and the two adjacent boundary
	squares.  Thus the four corner cards contribute \(4\cdot3=12\) boundary
	card-square incidences.  A non-corner boundary card touches exactly one side
	and covers exactly two boundary squares on that side, so the \(N_e\)
	non-corner boundary cards contribute \(2N_e\) boundary card-square incidences.
	Therefore
	\[
	s_{1,\mathrm{b}}+2((4m-4)-s_{1,\mathrm{b}})=12+2N_e.
	\]
	Solving for $s_{1,\mathrm{b}}$ gives
	\begin{equation}
		{s_{1,\mathrm{b}}=8m-20-2N_e.}
		\label{eq:Pb}
	\end{equation}
	
	The number of boundary cards is \(4+N_e\). Therefore the boundary private-square surplus which counts how many private squares remain on the boundaries after assigning boundary cards to private squares is exactly:
	\[
	s_{1,\mathrm{b}}-(4+N_e)=8m-24-3N_e.
	\]
	
	Let \(s_{1,\mathrm{int}}\) be the number of private squares in the interior \((m-2)\times(m-2)\) sub-grid, and define
	\[
	E=s_{1,\mathrm{int}}-N_{\mathrm{int}}.
	\]
	Every interior card has all four of its squares in the interior.  Therefore
	its required private square is also an interior square.  Since distinct cards
	cannot share a private square, the interior private squares assigned to
	interior cards are distinct, and hence $E\ge0.$
	Since $s_1=s_{1,\mathrm{b}}+s_{1,\mathrm{int}}$
	and $|U|=(4+N_e)+N_{\mathrm{int}}$,
	we obtain
	\begin{equation}
		{p_U=8m-24-3N_e+E.}
		\label{eq:p-boundary-cancellation}
	\end{equation}

	Substituting \eqref{eq:p-boundary-cancellation} into the defect identity gives
	\[
	5|U|
	=
	2m^2-(8m-24-3N_e+E)+s_3+2s_4.
	\]
	Thus
	\begin{equation}
		{
			5|U|
			=
			2m^2-8m+24+3N_e-E+s_3+2s_4.
		}
		\label{eq:first-general}
	\end{equation}
	
	To obtain a second relation involving the same quantities, we count the squares in the interior of the grid according to their multiplicities.
	
	The interior of the grid has \((m-2)^2\) squares. Boundary squares have multiplicity at most \(2\), so all squares in \(S_3\) and \(S_4\) lie in the interior. Similarly, let \(s_{2,\mathrm{int}}\) be the number of interior squares covered twice. Then
	\[
	s_{1,\mathrm{int}}+s_{2,\mathrm{int}}+s_3+s_4=(m-2)^2.
	\]
	Since
	\[
	s_{1,\mathrm{int}}=N_{\mathrm{int}}+E
	\]
	and
	\[
	N_{\mathrm{int}}=|U|-N_e-4,
	\]
	we get
	\[
	(|U|-N_e-4)+E+s_{2,\mathrm{int}}+s_3+s_4=(m-2)^2.
	\]
	Solving for \(|U|\),
	\begin{equation}
		{
			|U|
			=
			m^2-4m+8+N_e-E-s_{2,\mathrm{int}}-s_3-s_4.
		}
		\label{eq:second-general}
	\end{equation}
	
	The advantage of these two relations is that the \(s_3\)-terms occur with opposite signs. Adding them therefore eliminates the contribution of triple-covered squares and leaves an identity involving only the boundary term, the interior private-square surplus, the double-covered interior squares, and the fourfold-covered squares. This gives the following boundary-cancellation identity.
	
	\begin{theorem}[Boundary-cancellation identity]
		For every irredundant cover of \(G_m\), with \(m\ge3\),
		\begin{equation}
			{
				6|U|
				=
				3m^2-12m+32+4N_e-2E-s_{2,\mathrm{int}}+s_4.
			}
			\label{eq:cancellation-general}
		\end{equation}
	\end{theorem}
	
	\begin{proof}
		Add \eqref{eq:first-general} and \eqref{eq:second-general}:
		\[
		\begin{aligned}
			6|U|
			=
			\bigl(2m^2-8m+24+3N_e-E+s_3+2s_4\bigr)
			+
			\bigl(m^2-4m+8+N_e-E-s_{2,\mathrm{int}}-s_3-s_4\bigr)
		\end{aligned}
		\]
		Collecting terms gives
		\[
		6|U|
		=
		3m^2-12m+32+4N_e-2E
		-s_{2,\mathrm{int}}+(s_3-s_3)+(2s_4-s_4).
		\]
		The \(s_3\)-terms cancel, and the remaining \(S_4\)-contribution is \(s_4\).
		This is exactly \eqref{eq:cancellation-general}.
	\end{proof}
	
	This cancellation is the main algebraic advantage of separating boundary
	private-square surplus from interior area.  It turns the triple-covered squares
	into a cancelled variable and leaves the remaining task of controlling
	\(N_e\), \(E\), \(s_{2,\mathrm{int}}\), and \(s_4\).

	It remains to control the contribution of \(s_4\) in the boundary-cancellation identity. For this purpose, we now turn from the preceding counting argument to the local geometry around a fourfold-covered square. The irredundancy condition forces a particularly rigid configuration around every square in \(S_4\), and this local structure will allow us to bound how densely such squares can occur.

	\begin{lemma}[External cards meet a footprint corner]
		\label{lem:external-footprint-card}
		Let \(x=(i,j)\in S_4\).  If a \(2\times2\) card meets \(\Phi(x)\) but is not one
		of the four cards containing \(x\), then it covers at least one corner of
		\(\Phi(x)\). 
	\end{lemma}
	
	\begin{proof}
		Write the upper-left corner of the card as \((i+\alpha,j+\beta)\).  In order to
		meet \(\Phi(x)\), the offsets must satisfy
		\[
		\alpha,\beta\in\{-2,-1,0,1\}.
		\]
		The four internal cards containing \(x\) are exactly those with
		\[
		\alpha,\beta\in\{-1,0\}.
		\]
		Thus an external card has one of the twelve offset pairs listed below.  The
		last column gives a footprint corner covered by that card.
		\[
		\begin{array}{c|c}
			(\alpha,\beta) & \text{corner covered}\\
			\hline
			(-2,-2),\,(-2,-1),\,(-1,-2) & (i-1,j-1)\\
			(-2,0),\,(-2,1),\,(-1,1) & (i-1,j+1)\\
			(0,-2),\,(1,-2),\,(1,-1) & (i+1,j-1)\\
			(0,1),\,(1,0),\,(1,1) & (i+1,j+1)
		\end{array}
		\]
		This exhausts all possibilities and proves the claim.
	\end{proof}
	
	\begin{lemma}[Isolation of an \(S_4\)-footprint]
		\label{lem:s4-isolation}
		Let \(x\in S_4\).  No selected card other than the four cards containing \(x\)
		can meet \(\Phi(x)\).
	\end{lemma}
	
	\begin{proof}
		Write \(x=(i,j)\).  The four selected cards surrounding \(x\) cover the
		\(3\times3\) footprint with multiplicities shown in Figure~\ref{fig:s4-footprint}.
		For each of these four cards, all footprint squares except the corresponding
		corner are already covered by another one of the four surrounding cards.  Thus
		the only possible private squares of
		\[
		C_{i-1,j-1},\quad C_{i-1,j},\quad C_{i,j-1},\quad C_{i,j}
		\]
		are respectively
		\[
		(i-1,j-1),\quad (i-1,j+1),\quad (i+1,j-1),\quad (i+1,j+1).
		\]
		
		Suppose another selected card \(Q\) meets \(\Phi(x)\).  By
		Lemma~\ref{lem:external-footprint-card}, \(Q\) covers one of these four corner
		squares.  The surrounding card whose only possible private square is that
		corner would then have no private    square, contradicting irredundance.  Hence no
		such \(Q\) exists.
	\end{proof}
	
	\begin{corollary}[Disjointness of \(S_4\)-footprints]
		\label{cor:s4-footprints-disjoint}
		If \(x,y\in S_4\) and \(x\neq y\), then
		\[
		\Phi(x)\cap\Phi(y)=\varnothing.
		\]
	\end{corollary}
	
	\begin{proof}
		If \(\Phi(x)\) and \(\Phi(y)\) met, then at least one of the four selected cards
		containing \(y\) would meet \(\Phi(x)\). Since \(y\neq x\), those four cards
		cannot all be the four cards containing \(x\). Indeed, writing \(y=x+(r,s)\), the assumption \(\Phi(x)\cap\Phi(y)\neq\varnothing\) implies \(|r|,|s|\le2\), with \((r,s)\neq(0,0)\). If \(|r|,|s|\le1\), then \(y\in\Phi(x)\), so every card containing \(y\) meets \(\Phi(x)\), and at least one of them is not among the four cards containing \(x\). Otherwise, one of \(|r|\) or \(|s|\) equals \(2\); choosing a card containing \(y\) that extends toward \(\Phi(x)\) in that coordinate gives a card that meets \(\Phi(x)\) but cannot be one of the four cards containing \(x\).
		
		Thus, in either case, there exists a selected card containing \(y\) that meets \(\Phi(x)\) but is not one of the four cards containing \(x\). This contradicts Lemma~\ref{lem:s4-isolation}.
	\end{proof}

	\begin{lemma}[Packing bound for \(S_4\)-squares]
		For every irredundant cover of \(G_m\), with \(m\ge3\),
		\[
		s_4\le \left\lceil\frac{m-2}{3}\right\rceil^2.
		\]
	\end{lemma}
	
	\begin{proof}
		The possible centers of \(S_4\)-squares lie in the interior coordinate set
		\[
		\{2,\ldots,m-1\}\times\{2,\ldots,m-1\}.
		\]
		Partition each coordinate direction into consecutive intervals of length at
		most \(3\).  This gives \(\lceil(m-2)/3\rceil\) intervals in each direction and
		hence \(\lceil(m-2)/3\rceil^2\) rectangular boxes.  If two \(S_4\)-centers lay
		in the same box, their \(3\times3\) footprints would intersect.  This is
		impossible by Corollary~\ref{cor:s4-footprints-disjoint}.  Hence each box
		contains at most one \(S_4\)-center.
	\end{proof}
	
	We can now combine the preceding estimates. The boundary-card bound controls \(N_e\), while \(E\) and \(s_{2,\mathrm{int}}\) are nonnegative, and the footprint argument gives an upper bound for \(s_4\). Substituting these estimates into the boundary-cancellation identity yields our first boundary-based general upper bound.
	
	Using
	\[
	N_e\le4\left\lfloor\frac{2m}{3}\right\rfloor-8,
	\qquad
	E\ge0,
	\qquad
	s_{2,\mathrm{int}}\ge0,
	\]
	and
	\[
	s_4\le\left\lceil\frac{m-2}{3}\right\rceil^2
	\]
	in \eqref{eq:cancellation-general}, we obtain
	\[
	6|U|
	\le
	3m^2-12m+32
	+
	4\left(4\left\lfloor\frac{2m}{3}\right\rfloor-8\right)
	+
	\left\lceil\frac{m-2}{3}\right\rceil^2.
	\]
	Simplifying,
	\[
	6|U|
	\le
	3m^2-12m
	+
	16\left\lfloor\frac{2m}{3}\right\rfloor
	+
	\left\lceil\frac{m-2}{3}\right\rceil^2.
	\]
	We have proved the following theorem.
	
	\begin{theorem}[Boundary-cancellation packing bound]
		For every \(m\ge3\),
		\[
		{
			F(m)
			\le
			\left\lfloor
			\frac{
				3m^2-12m
				+
				16\left\lfloor\frac{2m}{3}\right\rfloor
				+
				\left\lceil\frac{m-2}{3}\right\rceil^2
			}{6}
			\right\rfloor.
		}
		\]
		Consequently,
		\[
		{
			F(m)\le \frac{14}{27}m^2-\frac{2}{9}m.
		}
		\]
		
	\end{theorem}
	
	\begin{proof}
		The displayed finite bound follows from the preceding derivation and the
		integrality of \(|U|\). To obtain the simpler consequence, first note that
		\[
		\left\lfloor\frac{2m}{3}\right\rfloor
		\le
		\frac{2m}{3}.
		\]
		Moreover, since \(m\) is an integer,
		\[
		\left\lceil\frac{m-2}{3}\right\rceil
		=
		\left\lfloor\frac{m}{3}\right\rfloor
		\le
		\frac{m}{3}.
		\]
		Therefore,
		\[
		\begin{aligned}
			F(m)
			&\le
			\frac{
				3m^2-12m
				+
				16\left(\frac{2m}{3}\right)
				+
				\left(\frac{m}{3}\right)^2
			}{6}
			\\
			&=
			\frac{
				\frac{28}{9}m^2-\frac{4}{3}m
			}{6}
			\\
			&=
			\frac{14}{27}m^2-\frac{2}{9}m.
		\end{aligned}
		\]
		Thus the displayed finite upper bound has leading term
		\(\frac{14}{27}m^2\) together with a negative correction of linear order.
	\end{proof}

	\subsection{A coupled \(S_4\) bound}
	
	The preceding estimate used only the packing consequence of isolated
	\(S_4\)-footprints.  The same footprints also force nearby double-covered
	interior squares.  This lets the term
	\(-s_{2,\mathrm{int}}+s_4\) in the boundary-cancellation identity be treated as
	a single nonpositive contribution.
	
	\begin{lemma}[Each \(S_4\)-square forces interior \(S_2\)-squares]
		For every irredundant cover of \(G_m\), with \(m\ge4\),
		\[
		{s_{2,\mathrm{int}}\ge 2s_4.}
		\]
	\end{lemma}
	
	\begin{proof}
		Let \(x=(i,j)\in S_4\).  By Lemma~\ref{lem:s4-isolation}, no selected card
		other than the four cards containing \(x\) meets \(\Phi(x)\).  Therefore the four edge-middle
		squares of the footprint,
		\[
		(i-1,j),\qquad (i,j-1),\qquad (i,j+1),\qquad (i+1,j),
		\]
		are covered exactly twice.  Hence they are \(S_2\)-squares.
		
		Since \(x\in S_4\), it is an interior square, so
		\[
		2\le i,j\le m-1.
		\]
		For \(m\ge4\), at least two of the four edge-middle squares above are interior
		grid squares.  Indeed, if one of the two vertical edge-middle squares lies on
		the boundary, then the other vertical edge-middle square is interior; and the
		same is true horizontally.  Thus each \(S_4\)-square contributes at least two
		interior \(S_2\)-squares.
		
		Finally, distinct \(S_4\)-squares have disjoint footprints by
		Corollary~\ref{cor:s4-footprints-disjoint}.  Therefore these contributed
		interior \(S_2\)-squares are distinct for distinct \(S_4\)-squares.  Hence
		$s_{2,\mathrm{int}}\ge 2s_4.$
	\end{proof}
	
	\begin{theorem}[Uniform general upper bound]
		For every \(m\ge4\),
		\[
		{
			F(m)
			\le
			\left\lfloor
			\frac{
				3m^2-12m
				+
				16\left\lfloor\frac{2m}{3}\right\rfloor
			}{6}
			\right\rfloor.
		}
		\]
		Consequently,
		\[
		F(m)\le\frac12m^2-\frac{2m}{9}.
		\]
	\end{theorem}
	
	\begin{proof}
		Start from the boundary-cancellation identity
		\[
		6|U|
		=
		3m^2-12m+32+4N_e-2E-s_{2,\mathrm{int}}+s_4.
		\]
		By the preceding lemma,
		\[
		s_{2,\mathrm{int}}\ge 2s_4.
		\]
		Therefore
		\[
		-s_{2,\mathrm{int}}+s_4\le -s_4\le0.
		\]
		Since \(E\ge0\), we obtain
		\[
		6|U|
		\le
		3m^2-12m+32+4N_e.
		\]
		Using the boundary-card estimate
		\[
		N_e\le4\left\lfloor\frac{2m}{3}\right\rfloor-8,
		\]
		we get
		\[
		6|U|
		\le
		3m^2-12m+32
		+
		4\left(
		4\left\lfloor\frac{2m}{3}\right\rfloor-8
		\right).
		\]
		The constant terms cancel, giving
		\[
		6|U|
		\le
		3m^2-12m
		+
		16\left\lfloor\frac{2m}{3}\right\rfloor.
		\]
		Dividing by \(6\) and using the integrality of \(|U|\) proves the finite bound.
		
		For the simpler consequence, use
		\[
		\left\lfloor\frac{2m}{3}\right\rfloor
		\le
		\frac{2m}{3}.
		\]
		It follows that
		\[
		\begin{aligned}
			6|U|
			&\le
			3m^2-12m
			+
			16\left(\frac{2m}{3}\right)\\
			&=
			3m^2-\frac{4m}{3}.
		\end{aligned}
		\]
		Dividing by \(6\), we obtain
		\[
		|U|\le\frac12m^2-\frac{2m}{9}.
		\]
		Since this holds for every irredundant cover,
		\[
		F(m)\le\frac12m^2-\frac{2m}{9}.
		\]
	\end{proof}

	\begin{remark}
		This bound improves the leading coefficient from \(\frac{14}{27}\)
		to \(\frac12\), matching the lower bound established in the next
		section.
	\end{remark}
	\section{Lower bounds}
	
	The preceding sections give upper bounds.  We begin with a simple product
	construction and then use a protected period-four staircase to obtain the
	correct leading density and a linear error term.
	
	\begin{theorem}[Parity lower bound]
		For every \(m\ge2\),
		\[
		F(m)\ge \left\lceil\frac m2\right\rceil^2.
		\]
		Consequently,
		\[
		F(m)\ge \frac14m^2,
		\]
		and \(\lceil m/2\rceil^2=\frac14m^2+O(m)\).
	\end{theorem}
	
	\begin{proof}
		Let
		\[
		A_m=
		\begin{cases}
			\{1,3,5,\ldots,m-1\}, & m \text{ even},\\
			\{1,3,5,\ldots,m-2,m-1\}, & m \text{ odd}.
		\end{cases}
		\]
		Then \(|A_m|=\lceil m/2\rceil\).  Select all cards \(C_{a,b}\) with
		\(a,b\in A_m\).
		
		The one-dimensional intervals \(\{a,a+1\}\), with \(a\in A_m\), cover
		\(\{1,\ldots,m\}\).  Hence the selected product cards cover \(G_m\).  Moreover,
		each interval \(\{a,a+1\}\) has a coordinate covered by no other interval in
		this family: for odd \(a<m-1\) one may use \(a\), while for the extra terminal
		choice \(a=m-1\) in the odd case one may use \(m\).  If \(t(a)\) denotes such a
		private coordinate for the row interval and \(t(b)\) one for the column
		interval, then the square \((t(a),t(b))\) is covered only by \(C_{a,b}\).  Thus
		every selected card has a private square.
	\end{proof}
	
	\begin{lemma}[The period-four staircase pattern]
		\label{lem:infinite-staircase}
		On the infinite square grid, select the \(2\times2\) cards with upper-left
		corners \((a,b)\in\mathbb Z^2\) satisfying
		\[
		b-a\equiv0 \text{ or }3\pmod4.
		\]
		Then every grid square is covered either once or three times.  Moreover, every
		selected card has a private square in this infinite periodic configuration.
	\end{lemma}
	
	\begin{proof}
		Consider the square \(x=(i,j)\).  The four possible cards containing \(x\) have
		upper-left corners
		\[
		(i-1,j-1),\quad(i-1,j),\quad(i,j-1),\quad(i,j).
		\]
		If \(d=j-i\), then the corresponding differences \(b-a\pmod4\) are
		\[
		d,\quad d+1,\quad d-1,\quad d.
		\]
		For \(d\equiv0,1,2,3\pmod4\), respectively, the numbers of selected cards are
		\(3,1,1,3\).  Hence every square is covered, with multiplicity either \(1\) or
		\(3\).
		
		If \(C_{a,b}\) is selected and \(b-a\equiv0\pmod4\), then \((a,b+1)\) is
		private for \(C_{a,b}\).  Indeed, the four card positions that could cover it
		have residue differences
		\[
		1,\quad2,\quad0,\quad1,
		\]
		so only \(C_{a,b}\) is selected.  If \(b-a\equiv3\pmod4\), then \((a+1,b)\)
		is private: the four relevant residue differences are
		\[
		2,\quad3,\quad1,\quad2,
		\]
		and again \(C_{a,b}\) is the unique selected card covering the displayed
		square.
	\end{proof}
	
	\begin{theorem}[Protected period-four lower bound]
		\label{thm:staircase-lower-bound}
		For every \(m\ge6\),
		\[
		F(m)\ge\frac12m^2-2m+3.
		\]
		Consequently, for every \(m\ge2\),
		\[
		F(m)\ge\frac12m^2-2m,
		\]
		and therefore
		\[
		\liminf_{m\to\infty}\frac{F(m)}{m^2}\ge\frac12.
		\]
	\end{theorem}
	
	\begin{proof}
		Assume first that \(m\ge6\), and put
		\[
		Q=\{3,\ldots,m-2\}\times\{3,\ldots,m-2\}.
		\]
		Retain exactly those staircase cards whose explicit private square from
		Lemma~\ref{lem:infinite-staircase} lies in \(Q\).  Equivalently, define
		\[
		\begin{aligned}
			\mathcal P
			={}&\{C_{i,j-1}:(i,j)\in Q,\ j-i\equiv1\pmod4\}\\
			&\mathbin{\cup}
			\{C_{i-1,j}:(i,j)\in Q,\ j-i\equiv2\pmod4\}.
		\end{aligned}
		\]
		All these cards lie on the board; in fact, each of their squares lies in
		\(\{2,\ldots,m-1\}^2\).  For the card associated with \((i,j)\), choose
		\((i,j)\) as its private square.  Lemma~\ref{lem:infinite-staircase} shows that
		these chosen squares are private within \(\mathcal P\).
		
		We claim that \(\mathcal P\) covers \(Q\).  Let \(x=(i,j)\in Q\), and write
		\(d=j-i\pmod4\).  If \(d=1\) or \(d=2\), then \(x\) itself is one of the
		chosen private squares and is therefore covered.  If \(d=0\), the three
		selected staircase cards containing \(x\) have chosen private squares
		\[
		(i-1,j),\qquad(i+1,j-1),\qquad(i,j+1).
		\]
		At least one lies in \(Q\).  Indeed, unless \(x\) is on both the top and right
		sides of \(Q\), one of the first and third squares lies in \(Q\); at the
		top-right corner, the middle square lies in \(Q\).  Thus at least one of the
		three cards belongs to \(\mathcal P\).
		
		Similarly, if \(d=3\), the relevant private squares are
		\[
		(i,j-1),\qquad(i-1,j+1),\qquad(i+1,j).
		\]
		At least one lies in \(Q\); at the only exceptional position, the bottom-left
		corner, the middle square lies in \(Q\).  Hence \(\mathcal P\) covers all of
		\(Q\).
		
		The complement of \(Q\) is the disjoint union of the four rectangles
		\[
		\begin{aligned}
			T&=\{1,2\}\times\{1,\ldots,m\},\\
			B&=\{m-1,m\}\times\{1,\ldots,m\},\\
			L&=\{3,\ldots,m-2\}\times\{1,2\},\\
			R&=\{3,\ldots,m-2\}\times\{m-1,m\}.
		\end{aligned}
		\]
		Each rectangle has both side lengths at least \(2\).  Cover each coordinate
		interval by length-\(2\) intervals starting at every other position, adding a
		terminal interval when its length is odd, and take Cartesian products.  Let
		\(\mathcal R\) be the union of the resulting four rectangular parity covers.
		Then \(\mathcal R\) covers \(G_m\setminus Q\), and every repair card is
		contained in \(G_m\setminus Q\).  Consequently, no repair card covers a chosen
		private square of \(\mathcal P\).
		
		Every repair card also has a private square in the combined family.  For a card
		in \(T\), use row \(1\) together with the private coordinate of its
		one-dimensional column interval.  For cards in \(B,L,R\), use respectively row
		\(m\), column \(1\), and column \(m\), together with the corresponding private
		interval coordinate.  These squares are private within their rectangular
		parity covers.  The four rectangles are disjoint, and no card of
		\(\mathcal P\) reaches the outer boundary.  Hence
		\[
		\mathcal U=\mathcal P\cup\mathcal R
		\]
		is an irredundant cover of \(G_m\).
		
		Figure~\ref{fig:protected-period-four} illustrates the construction
		for \(m=8\).
		
		\begin{figure}[htbp]
			\centering
			
			\includegraphics[width=0.78\linewidth]{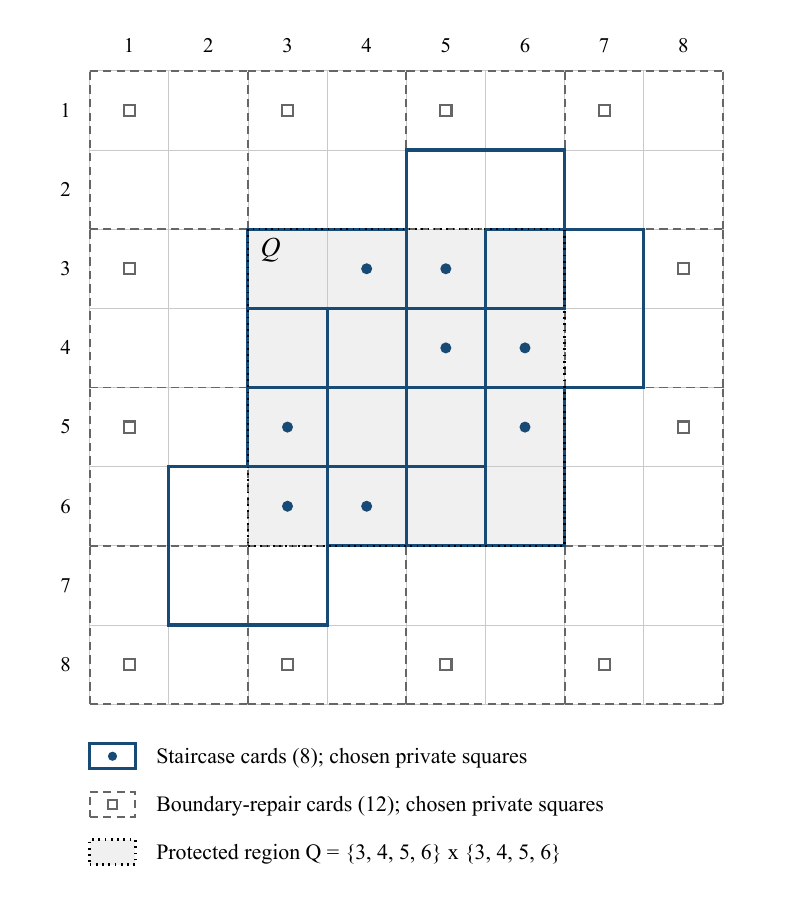}
			\caption{The protected period-four construction for \(m=8\),
				using eight staircase cards and twelve boundary-repair cards.
				The marked private squares lie in \(Q\) and on the outer
				boundary, respectively.}
			\label{fig:protected-period-four}
		\end{figure}
		
		It remains to count cards.  Put \(n=m-4\), the side length of \(Q\).  The map
		from \(\mathcal P\) to its chosen private squares is a bijection onto
		\[
		\{(i,j)\in Q:j-i\equiv1\text{ or }2\pmod4\}.
		\]
		For \(t\in\{0,1,2,3\}\), let \(c_t\) denote the number of integers in \(\{1,\ldots,n\}\) congruent to \(t\pmod 4\). The number of pairs \((i,j)\) satisfying \(j-i\equiv1\) or \(2\pmod4\) is
		
		$$
		\sum_{t=0}^{3} c_t\bigl(c_{t+1}+c_{t+2}\bigr),
		$$
		
		where the subscripts are taken modulo \(4\). Writing \(n=4q+r\) and substituting the corresponding residue-class sizes gives
		
		\[
		|\mathcal P|
		=
		\begin{cases}
			n^2/2, & n\equiv0\pmod4,\\
			(n^2-1)/2, & n\equiv1\text{ or }3\pmod4,\\
			n^2/2-1, & n\equiv2\pmod4.
		\end{cases}
		\]
		Therefore
		\[
		|\mathcal P|
		\ge\frac{(m-4)^2}{2}-1
		=\frac12m^2-4m+7.
		\]
		
		The top and bottom rectangles each use \(\lceil m/2\rceil\) cards, while the
		left and right rectangles each use \(\lceil(m-4)/2\rceil\) cards.  Thus
		\[
		|\mathcal R|
		=2\left\lceil\frac m2\right\rceil
		+2\left\lceil\frac{m-4}{2}\right\rceil
		\ge2m-4.
		\]
		The two families are disjoint, so
		\[
		F(m)
		\ge|\mathcal U|
		=|\mathcal P|+|\mathcal R|
		\ge\frac12m^2-2m+3.
		\]
		
		For \(2\le m\le5\), the parity lower bound verifies
		\(F(m)\ge\frac12m^2-2m\) directly.
	\end{proof}

	Together with the uniform upper bound, Theorem~\ref{thm:staircase-lower-bound}
	gives, for \(m\ge6\),
	\[
	\frac{2m}{9}
	\le
	\frac12m^2-F(m)
	\le
	2m-3.
	\]
	Consequently,
	\[
	{F(m)=\frac12m^2-\Theta(m)}
	\]
	and
	\[
	\lim_{m\to\infty}\frac{F(m)}{m^2}=\frac12.
	\]

	\subsection*{Main open problem}
	
	The preceding bounds show that the deficit
	\[
	\frac12m^2-F(m)
	\]
	has linear order.  The remaining problem is to determine its finer asymptotic
	behaviour.  In particular, decide whether there is a constant \(c>0\) such
	that
	\[
	F(m)=\frac12m^2-cm+O(1),
	\]
	or, more generally, whether there are constants \(c_0,c_1,c_2,c_3>0\) such
	that
	\[
	F(m)=\frac12m^2-c_r m+O(1)
	\qquad(m\equiv r\pmod4).
	\]
	The period-four staircase and boundary-repair constructions suggest that the
	linear term may contain residue-class effects modulo \(4\).
	
	For the \(10\times10\) benchmark, the defect identity shows that the upper
	bound \(|U|\le39\) would follow from proving \(D_U>0\).  The boundary-cancellation
	identity specializes to
	\[
	6|U|=212+4N_e-2E-s_{2,\mathrm{int}}+s_4.
	\]
	Thus a non-computational proof of the benchmark upper bound would follow from
	a finite coupling inequality strong enough to force
	\[
	2E+s_{2,\mathrm{int}}-s_4\ge4N_e-22.
	\]
	This indicates the boundary-interior interaction that a sharper general theory
	would need to capture in order to determine the coefficient of the linear
	deficit.
	
	\appendix
	\section{An auxiliary mex-level bound}
	\label{app:mex}

	\begin{definition}
		For a finite set \(A\subseteq\mathbb Z_{>0}\), the smallest integer in $\mathbb{Z}_{>0} \setminus A$ is denoted by $\mex_+(A),$ that is,
		\[
		\mex_+(A)=\min\{k\in\mathbb Z_{>0}:k\notin A\}.
		\]
	\end{definition}
	
	
	We now introduce a height-based decomposition of the card set. For a given irredundant cover \( U=\{C^1,C^2\ldots,C^{|U|}\}\), we fix an arbitrary ordering of its cards. Since the ordering is arbitrary, any selected card can be placed first and assigned level \(1\). Thus, we select a card \(C^1\) from \(U\) and assign \(\ell(C^1)=1\), meaning that its level is \(1\). For each remaining card \(C^k\), where \(k\geq2\), we determine its level recursively by
	\[
	\ell(C^k)
	=
	mex_+
	\{\ell(C^i):1\le i<k,\ C^i\cap C^k\neq\varnothing\}.
	\]
	
	For each \(i\ge1\), we define the level class
	\[
	L_i=\{C^k\in U:\ell(C^k)=i\}.
	\]

	The maximum level of the decomposition can be written as $
	h=\displaystyle \max_{C^k\in U}\ell(C^k).$
	Note that this decomposition depends on the assigned ordering. This dynamic partitioning will be used to obtain the first bound in Section \ref{section_6}.
	
	In this section, we develop our first general upper bound for \(F(m)\) by introducing an auxiliary decomposition of the selected cards. We first assign levels to the cards using the positive mex operation and study the resulting structure. In particular, we establish a local bound on the number of cards that can overlap a given selected card and use it to show that the decomposition has height at most five. We then use this five-level structure to obtain a first global counting bound on the size of an irredundant cover. Although the resulting estimate is weaker than the bounds derived later from square multiplicities and boundary effects, it provides a simple first illustration of how the irredundancy condition imposes global restrictions on the number of selected cards.
	
	\subsection{Card levels and the height bound}
	
	For \(C^k\in U\), let
	\[
	\mathcal N(C^k)=\{C^l\in U\setminus\{C^k\}: C^k\cap C^l\neq\varnothing\}.
	\]
	Equivalently, $\mathcal N(C^k)$ is the set of selected cards, other than $C^k$, that share at least one square with $C^k$.

	\begin{lemma}[Local neighborhood bound]
		For every selected card \(C^k\in U\), $|\mathcal N(C^k)|\le 4.$
		\label{Local neighborhood bound}
	\end{lemma}
	
	\begin{proof}
		Let \(C^k=C_{a,b}\). By symmetry, assume that its upper-left square
		\((a,b)\) is private. We may work in the infinite grid, since the boundary can
		only remove possible neighbours. Every card meeting \(C^k\) has the form $C_{a+\alpha,b+\beta},$ where $(\alpha,\beta)\neq (0,0).$
		The privacy of \((a,b)\) excludes the three offsets, where $(\alpha,\beta)$ is equal to one of the pairs $(-1,-1), (-1,0), (0,-1),$ so at most five neighbouring positions remain. These five positions cannot all be selected. Indeed, if they were, then
		\(C_{a,b+1}\) would have no private square: its left column is covered by
		\(C_{a,b}\), while its upper-right and lower-right squares are covered by
		\(C_{a-1,b+1}\) and \(C_{a+1,b+1}\), respectively. This contradicts
		irredundance. Hence at least one of the five positions is absent, and therefore
		\[
		|\mathcal N(C^k)|\le 4.
		\]
	\end{proof}

	\begin{theorem}(Height)
		For every irredundant cover \(U\), the height \(h\) of the mex-defined level decomposition satisfies \(h\le5\).
	\end{theorem}
	
	\begin{proof}
		By construction, any two distinct cards which overlap each other are in different levels. Indeed, if two overlapping cards were assigned the same level, then the card appearing later in the chosen ordering would see that level among its earlier overlapping neighbours, so the positive mex rule could not assign it the same level. 
		
		Suppose a card \(C^k\) is on level \(i\), that is, \(\ell(C^k)=i\). By the definition of \(\mex_+\), the card \(C^k\) must overlap at least one earlier card of each level $1,2,\ldots,i-1.$ Thus $|\mathcal N(C^k)|\ge i-1.$ By Lemma \ref{Local neighborhood bound}, $|\mathcal N(C^k)|\le4,$
		therefore the level $i$ must be less than or equal to $5.$ This holds for every selected card \(C^k\) and we obtain that the maximum level $h$ must be less than or equal to $5.$
	\end{proof}
	
	\subsection{A first general layer-counting bound}
	
	\label{section_6}
	We now prove the first direct consequence of the height theorem. This bound is weaker than the later boundary bounds, but it is conceptually important because it is the cleanest place where the five-level structure gives a global estimate.
	
	\begin{theorem}
		For every \(m\ge2\),
		\[
		{
			F(m)\le \left\lfloor \frac{5m^2}{8}\right\rfloor.
		}
		\]
		In particular,
		\[
		F(10)\le62.
		\]
	\end{theorem}
	
	\begin{proof}
		We showed earlier that \(h\le5\), hence the selected cards can be partitioned into at most five classes
		\[
		L_1,L_2,L_3,L_4,L_5,
		\]
		where cards in the same class are pairwise disjoint. Empty classes are allowed. Let $n_i=|L_i|.$ Then we know that the following equation holds:
		
		\begin{equation}
			n_1+n_2+n_3+n_4+n_5=|U|.
			\label{eq:level-sum}
		\end{equation}
		
		Relabel the classes, if necessary, so that \(L_1\) is a largest class. Since \(L_1,\ldots,L_5\) form a partition of \(U\), equation \eqref{eq:level-sum} gives
		
		$$
		n_1\ge \frac{|U|}{5}.
		$$
		
		Cards in \(L_1\) are pairwise disjoint. Therefore the union of the cards in
		\(L_1\) occupies exactly \(4n_1\) distinct grid squares. Any private square assigned to a card outside \(L_1\) must lie outside the union of the cards in \(L_1\), since every square in this union is already covered by a card of \(L_1\).
		Since every card \(C^i\) on every level needs a private square, and private squares assigned to distinct cards are necessarily distinct, we can write the following chain of inequalities in layer-by-layer form by assigning private squares successively to the remaining available squares as follows:

		\[
		n_2\le m^2-4n_1,
		\]
		\[
		n_3\le m^2-4n_1-n_2,
		\]
		\[
		n_4\le m^2-4n_1-n_2-n_3,
		\]
		\[
		n_5\le m^2-4n_1-n_2-n_3-n_4.
		\]
		The final inequality gives
		\[
		0\le 
		m^2-3n_1-|U|,
		\]
		and so again
		\[
		|U|\le m^2-3n_1\le m^2-\frac{3|U|}{5}.
		\]
		Rearranging the last inequality gives
		\[
		|U|+\frac{3|U|}{5}\le m^2,
		\]
		and hence
		\[
		\frac{8|U|}{5}\le m^2.
		\]
		Therefore
		\[
		|U|\le \frac{5m^2}{8}.
		\]
		Since \(|U|\) is an integer and the inequality holds for every irredundant cover,
		we conclude that
		\[
		F(m)\le \left\lfloor\frac{5m^2}{8}\right\rfloor.
		\]
		
	\end{proof}
	
	\bibliographystyle{plain}
	\bibliography{references}
\end{document}